%% file: main.tex
\documentclass[11pt]{amsart}
\usepackage[margin=1in]{geometry}
\usepackage[utf8]{inputenc}
\usepackage{amsmath}
\usepackage{amsfonts}
\usepackage{amsthm}
\usepackage{enumerate}
\usepackage{comment}
\usepackage{tikz,pgf,pgfplots}
\pgfplotsset{compat=newest}

\title{Hyperbolicity of Appell Polynomials of Functions in the \(\delta\)-Laguerre-Pólya Class}
\author{Jonas Iskander, Vanshika Jain}
\date{January 2020}

\newcommand{\R}{\mathbb{R}}
\newcommand{\C}{\mathbb{C}}

\newcommand{\LP}{\mathcal{LP}}

\newcommand{\Addresses}{{
  \bigskip
  \footnotesize

  \noindent J.~Iskander, \textsc{2046 Deren Way NE, Atlanta, GA 30345}\par\nopagebreak
  \textit{Email address}: \texttt{jonasiskander@gmail.com}

  \medskip

  \noindent V.~Jain, \textsc{Department of Mathematics, Massachusetts Institute of Technology, Cambridge, MA 02139}\par\nopagebreak
  \textit{Email address}: \texttt{vanshika@mit.edu}

}}

\newtheorem{theorem}{Theorem}[section]
\newtheorem{lemma}[theorem]{Lemma}
\newtheorem{corollary}[theorem]{Corollary}

\theoremstyle{definition}

\begin{document}

\maketitle

\begin{abstract}
We present a method for proving that Jensen polynomials associated with functions in the \(\delta\)-Laguerre-Pólya class have all real roots, and demonstrate how it can be used to construct new functions belonging to the Laguerre-Pólya class. As an application, we confirm a conjecture of Ono, which asserts that the Jensen polynomials associated with the first term of the Hardy-Ramanujan-Rademacher series formula for the partition function are always hyperbolic.
\end{abstract}

\section{Introduction and Statement of Results}

A polynomial of degree \(d\) is called \textit{hyperbolic} if it has \(d\) real roots, counting multiplicity. A function \(f(t)\) is said to be in the Laguerre-Pólya class (denoted \(\LP\)) if there exists a sequence of hyperbolic polynomials converging locally uniformly to \(f(t)\) in \(\C\) \cite{jensen1913}. The \(\LP\) class consists of functions of the form \begin{equation}\label{laguerre-polya-function-form}
    ct^me^{\alpha t - \beta t^2} \prod_{n = 1}^\infty \left(1 - \frac{t}{t_n}\right) e^{\frac{t}{t_n}}
\end{equation} for \(c, \alpha \in \R\), \(\beta \geq 0\), and \(\{t_n\}_1^\infty\) such that \(\sum_{k=1}^\infty |t_k|^{-2}\) converges, where \(\{t_k\}_1^\infty\) is allowed to have finitely many terms \cite{laguerre1898oeuvres}. It turns out that there is a simple criterion for confirming whether a function \(f(t)\) can be written in form \eqref{laguerre-polya-function-form}. Specifically, J.\@ Jensen proved that a function \(f(t)\) belongs to \(\LP\) if and only if for some \(t \in \R\), the \textit{Appell polynomials} \begin{equation}
    A_d^f(t; x) := e^{-xt} D_x^d(e^{xt}f(t)) = \sum_{k=0}^d \binom{d}{k} f^{(k)}(t) x^{d-k} 
\end{equation} are hyperbolic for all \(d \geq 0\) \cite{baricz2018zeros, jensen1913}.

The Laguerre-Pólya class of entire functions was initially studied in the late nineteenth century, but it has been of interest since the twentieth century because of its relation to the Riemann hypothesis. In particular, J.\@ Jensen and G.\@ Pólya proved that the Riemann hypothesis is true if and only if the Riemann xi-function belongs to \(\LP\) \cite{polya1927algebraisch}. As a result, there has been interest in expanding the mathematical toolkit for proving hyperbolicity for the Appell polynomials of functions which are not known to be in the \(\LP\) class. Due to their convenience when studying combinatorial sequences, \textit{Jensen polynomials} of the form \begin{equation}
    J_d^{a,n}(x) := \sum_{k=0}^d \binom{d}{k} a_{n+k}x^k
\end{equation} associated with various \(\{a_n\}_0^\infty\), which are closely related to Appell polynomials, have received much attention.

It is a classical result that if a sequence of polynomials of successive degrees is orthogonal with respect to a positive definite metric, then the polynomials are hyperbolic. Hermite polynomials are an example of such a sequence. P.\@ Turán strengthened this result by showing that if an expansion of a polynomial \(p(x)\) in orthogonal polynomials satisfies a simple criterion, then \(p(x)\) is hyperbolic \cite{turan}. In a recent paper, M.\@ Griffin, K.\@ Ono, L.\@ Rolen, and D.\@ Zagier studied Jensen polynomials of fixed degree associated with sequences satisfying a certain limiting behavior. They found that when suitably renormalized, the Jensen polynomials approach Hermite polynomials and hence eventually become hyperbolic \cite{griffin2019jensen}. In particular, the Taylor coefficients about \(1/2\) of the Riemann xi-function satisfy this condition. Effective bounds for when these Jensen polynomials become hyperbolic in the case of the Riemann xi-function are given in \cite{griffin2019jensenxi}.

The Jensen polynomials associated with the \textit{partition function} \(p(n)\), which is defined by its generating function \begin{equation*}
    \sum_{n=0}^\infty p(n) x^n := \prod_{k=1}^\infty \frac{1}{1-x^k},
\end{equation*} have also received much attention. J.\@ Nicolas \cite{nicolas} and S.\@ DeSalvo and I.\@ Pak \cite{DeSalvo} independently proved that the Jensen polynomials \(J_2^{p,n}(x)\) are hyperbolic for \(n \geq 25\). W.\@ Chen, D.\@ Jia, and L.\@ Wang conjectured that for any fixed degree \(d\), \(J_d^{p,n}(x)\) is hyperbolic for sufficiently large \(n\), and they proved eventual hyperbolicity in the case of \(d = 3\) \cite{chen}. H.\@ Larson and I.\@ Wagner independently gave bounds for arbitrary \(d\) and found the minimal shifts \(n\) after which the polynomials are hyperbolic for \(d \in \{3, 4, 5\}\) \cite{larson}.

An important tool for studying the Jensen polynomials of the partition function has been the Hardy-Ramanujan-Rademacher series expansion, \begin{equation}\label{partition-expansion}
    p(n) = 24 \sum_{k=1}^{\infty} A_k(n) \left(\frac{\pi}{12k}\right)^{5/2} C_{\frac{3}{2}}\left(\frac{\pi^2}{6k^2} \left(n-\frac{1}{24}\right)\right),
\end{equation} where \(A_k(n)\) is a Kloosterman sum and \(C_\nu(t)\) denotes the \textit{Bessel-Clifford} function, given by \begin{equation}
    C_\nu(t) := \sum_{k=0}^\infty \frac{1}{\Gamma(\nu+k+1)}\frac{t^k}{k!} = \begin{cases}
        |t|^{-\frac{\nu}{2}} I_\nu\big(2\sqrt{|t|}\big), \quad t \geq 0, \\
        |t|^{-\frac{\nu}{2}} J_\nu\big(2\sqrt{|t|}\big), \quad t \leq 0.
    \end{cases}
\end{equation} In a recent paper, H.\@ Chan and L.\@ Wang defined the \textit{fractional partition function} for \(\alpha \in \R\) in terms of its generating function, \begin{equation}
    \sum_{n=0}^\infty p_\alpha(n) x^n := \prod_{k=1}^\infty \frac{1}{(1-x^k)^\alpha},
\end{equation} which agrees with the ordinary partition function when \(\alpha = 1\) \cite{chan}. J.\@ Iskander, V.\@ Jain, and V.\@ Talvola derived an infinite series expansion analogous to \eqref{partition-expansion} for \(p_\alpha(n)\) when \(\alpha > 0\) \cite{exactformulae}.

Recently, K.\@ Ono conjectured that the Jensen polynomials \(J_d^{R_\alpha, n}(x)\) of the first term of the analogue of the Hardy-Ramanujan-Rademacher expansion for \(p_\alpha(n)\), \begin{equation}\label{rseries}
    R_\alpha(n) = 2\pi\left(\frac{\pi\alpha}{12}\right)^{\frac{\alpha}{2}+1} C_{\frac{\alpha}{2}+1}\left(\frac{\pi^2\alpha}{6}\left(n - \frac{\alpha}{24}\right)\right),
\end{equation} are hyperbolic for all \(n \geq \alpha/24\). Ono expected a similar result to hold for the first terms of analogous series expansions for generic weakly holomorphic modular forms with poles at the cusp infinity. 

In this paper, we prove Ono's conjecture for \(0 < \alpha < 3/2\) and justify his broader hypothesis by developing a theory of \(\delta\)-variants of hyperbolicity and the Laguerre-Pólya class. More precisely, we say that a polynomial is \textit{\(\delta\)-hyperbolic} if it is identically zero or its roots are all real, simple, and separated by at least \(\delta\), and we say that a function \(f(t)\) belongs to the \textit{\(\delta\)-Laguerre-Pólya} (\(\delta\)-\(\LP\)) class if there exists a sequence of \(\delta\)-hyperbolic polynomials converging locally uniformly to \(f(t)\) in \(\C\). The finite difference operator \(\Delta_{\delta,t}\) is defined by \begin{equation}
    \Delta_{\delta,t} f(t) := \frac{f(t+\delta) - f(t)}{\delta}.
\end{equation} For an arbitrary function \(f(t)\) and \(\delta > 0\), we define the \textit{\(\delta\)-Appell polynomials} associated to \(f\) by \begin{equation}
    A_d^{f,\delta}(t; x) := e^{-x t} \Delta_{\delta,t}^d(e^{x t} f(t)) = \frac{1}{\delta^d}\sum_{k=0}^d \binom{d}{k} (-1)^{d-k} f(t+k\delta) e^{k\delta x},
\end{equation} so that \(A_{d+1}^{f, \delta}(t; x) = e^{\delta x} A_d^{f, \delta}(t + \delta; x) - A_d^{f, \delta}(t; x)\) for all \(d \geq 0\).

Using these definitions, we obtain the following theorem.

\begin{theorem} \label{main-theorem}
Let \(f(t)\) be any function in the \(\delta\)-Laguerre-Pólya class, and let \(d \geq 0\) and \(t_0 \in \R\) be such that \(f(t)\) has no zeros on the interval \([t_0, t_0+d\delta]\). Then \(A_d^{f,\delta}(t_0; x)\) has \(d\) real zeros.
\end{theorem}

\noindent 
Theorem \ref{main-theorem} gives a method for producing functions in the \(\LP\) class.

\begin{theorem} \label{new-LP-functions}
Let \(f(t)\) be any function in the \(\delta\)-Laguerre-Pólya class, and suppose \(t_0 \in \R\) is such that \(f(t)\) has no zeros on \([t_0, \infty)\). Then the function \begin{equation*}
    g(x) := \sum_{k=0}^\infty f(t_0+k\delta)\frac{x^k}{k!}
\end{equation*} is entire and belongs to the Laguerre-Pólya class.
\end{theorem}

\noindent 
An immediate consequence of Theorem \ref{main-theorem} is the following, which states that the \(\delta\)-Appell polynomials \(A_d^{C_\nu,\delta}(x; t)\) are hyperbolic for suitable \(\nu\), \(\delta\), and \(d\).

\begin{theorem} \label{bessel-clifford-hyperbolicity}
For \(\nu \geq \frac{1}{2}\), \(0 < \delta \leq \frac{\pi^2}{4}\), and \(t \geq 0\), the \(\delta\)-Appell polynomials \(A_d^{C_\nu,\delta}(x; t)\) are hyperbolic for all \(d \geq 0\).
\end{theorem}

\noindent
This theorem directly implies Ono's conjecture for \(R_\alpha(n)\) given in equation \eqref{laguerre-polya-function-form} when \(0 < \alpha \leq 3/2\). 

\begin{corollary} \label{ono-conjecture-corollary}
The Jensen polynomials \(J_d^{R_\alpha, n}(x)\) are hyperbolic for \(0 < \alpha < 3/2\), \(n \geq \frac{\alpha}{24}\), and \(d \geq 0\). In particular, the Jensen polynomials associated with the first term of the Hardy-Ramanujan-Rademacher expansion for the partition function are always hyperbolic.
\end{corollary}

\subsection*{Acknowledgements}
The authors would like to thank Ken Ono, Larry Rolen, and Ian Wagner for their guidance. The research was supported by the generosity of the Asa Griggs Candler Fund, the National Security Agency under grant H98230-19-1-0013, and the National Science Foundation under grants 1557960 and 1849959.


\section{Preliminaries}

The motivation for the proof of our main result is derived in large part by the observation that the discrete difference operator \(\Delta_{\delta,t}\) preserves \(\delta\)-hyperbolicity. This is stated precisely in next theorem.

\begin{theorem} \label{delta-difference-theorem}
    Let \(\delta > 0\) and \(x \in \R\), and suppose that \(f(t)\) is a \(\delta\)-hyperbolic polynomial with roots \(\{t_k\}_1^d\) arranged in descending order. Then the polynomial \(g(t) := e^{\delta x} f(t+\delta) - f(t)\) is \(\delta\)-hyperbolic. Moreover, for all \(1 \leq k \leq d-1\), \(g(t)\) has a root in the interval \([t_{k+1}, t_k-\delta]\), and if \(x > 0\), then \(g(t)\) has a root in \((-\infty, t_d-\delta)\), whereas if \(x < 0\), then \(g(t)\) has a root in \((t_1, \infty)\).
\end{theorem}

\begin{proof}
If \(f(t)\) is constant, then \(g(t)\) is constant and hence \(\delta\)-hyperbolic. Otherwise, let \(\sigma\) be the sign of the leading coefficient of \(f(t)\). Then for all \(k\), we have \(-\sigma(-1)^k g(t_k) = -\sigma(-1)^k e^{\delta x}f(t_k+\delta) \geq 0\) and \(-\sigma(-1)^k g(t_k-\delta) = \sigma(-1)^k f(t_k-\delta) \geq 0\). Hence, there exist \(\{s_k\}_1^{d-1}\) such that \(g(s_k) = 0\) and \(t_{k+1} \leq s_k \leq t_k - \delta\) for \(1 \leq k \leq d - 1\). If \(x = 0\), then the degree of \(g(t)\) is \(d - 1\) and the \(s_k\) enumerate all the roots of \(g(t)\). If \(x \neq 0\), then there exists a \(s_d\) with either \(s_d < t_d - \delta\) in the case \(x > 0\) or \(s_d > t_1\) in the case \(x < 0\) such that \(g(s_d) = 0\), completing the proof.
\end{proof}

\begin{center}
\begin{figure}[]
  \input{delta-difference-theorem}
  \caption{The polynomials \(e^{\delta x}f(t+\delta)\), shown as a solid line, and \(-f(t)\), shown as a dashed line, for a particular choice of \(x\), \(\delta\), and \(f(t)\). Note that at the points where one of the polynomials is zero, the sign of \(g(t) := e^{\delta x}f(t+\delta) - f(t)\) is easily determined.}
  \label{fig:delta-difference-theorem}
\end{figure}
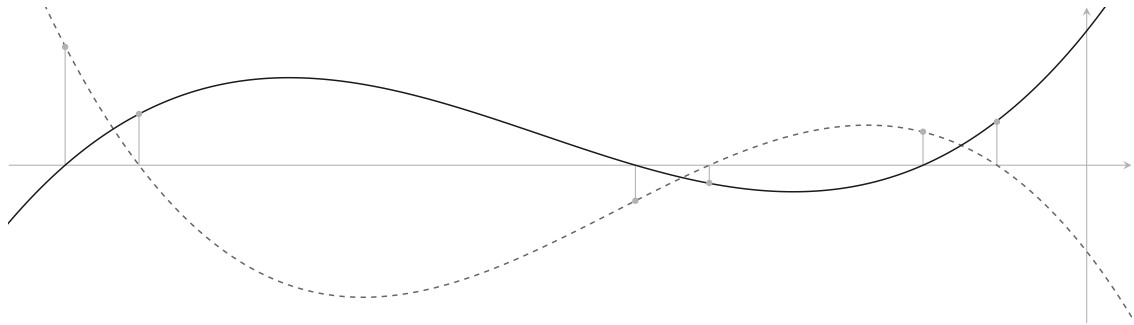
\end{center}

\noindent
By applying this theorem \(d\) times, one sees that for fixed \(x \in \R\), the polynomials \(A_d^{f,\delta}(t; x)\) in \(t\) are \(\delta\)-hyperbolic with roots interlacing those of \(A_{d-1}^{f,\delta}(t; x)\) in a specific way.

Many properties of \(\delta\)-hyperbolic polynomials carry over to functions in the \(\delta\)-\(\LP\) class. For example, the following corollary to Hurwitz's theorem states that functions in the \(\delta\)-\(\LP\) class have all real zeros separated by at least \(\delta\). In particular, this implies an analogue of Theorem \ref{delta-difference-theorem} for the \(\delta\)-\(\LP\) class---a result which does not hold for arbitrary functions with zeros separated by \(\delta\).

\begin{theorem}[A.\@ Hurwitz]
Let \(\{f_n(t)\}_0^\infty\) be a sequence of entire functions that converge locally uniformly to a holomorphic function \(f(t)\) which is not identically zero, and suppose that \(f(t)\) has a zero of multiplicity \(m\) at \(t_0\). Then for every \(\rho > 0\), there exists an \(N\) such that for all \(n \geq N\), \(f_n(t)\) has \(m\) zeroes in the disk \(|t - t_0| < \rho\), counting multiplicity.
\end{theorem}

\begin{corollary}
If \(f(t)\) is a nonzero function in the \(\delta\)-\(\LP\) class, then the zeros of \(f(t)\) are all real, simple, and separated by at least \(\delta\).
\end{corollary}

\begin{proof}
Since \(f(t)\) is in the \(\LP\) class, it must have only real zeros. Now let \(\{f_n(t)\}_1^\infty\) be a sequence of \(\delta\)-hyperbolic polynomials converging locally uniformly to \(f(t)\). To prove simplicity of the zeros, suppose \(t_0\) is a multiple zero of \(f(t)\). Then by Hurwitz's theorem, there exists an \(N \geq 0\) such that for all \(n \geq N\), the polynomials \(f_n(t)\) have two zeros, counting multiplicity, within the disk \(|t - t_0| < \delta/2\), a contradiction. Now suppose \(t_1 < t_2\) are two simple zeros of \(f(t)\) satisfying \(\varepsilon := t_2 - t_1 < \delta\), and set \(\rho := (\delta - \varepsilon)/2\). Again, by Hurwitz's theorem, there exist \(N_1, N_2 \geq 0\) such that for all \(n \geq N_1\), \(f_n(t)\) has a zero in the disk \(|t - t_1| < \rho\), and for all \(n \geq N_2\), \(f_n(t)\) has a zero in the disk \(|t - t_2| < \rho\). Consequently, we find that for \(n \geq \max\{N_1, N_2\}\), \(f_n(t)\) has two zeros separated by less than \(\varepsilon + 2\rho = \delta\), a contradiction.
\end{proof}

In order to avoid extra casework in the proof of our main theorem, we will find it useful to write functions in the \(\delta\)-\(\LP\) class as limits of \(\delta\)-hyperbolic polynomials of increasing degree.

\begin{lemma} \label{laguerre-polya-sequence-lemma}
Let \(f(t)\) be any function in the \(\delta\)-\(\LP\) class. Then there exists a sequence \(\{f_d(t)\}_0^\infty\) of \(\delta\)-hyperbolic polynomials converging locally uniformly to \(f(t)\) with \(\deg{f_{d+1}(t)} > \deg{f_d(t)}\) for all \(d\).
\end{lemma}

\begin{proof}
Let \(\{g_d(t)\}_0^\infty\) be a sequence of \(\delta\)-hyperbolic polynomials converging locally uniformly to \(f(t)\). If \(\limsup_{d \to \infty} \deg{g_d(t)} = +\infty\), we can construct a subsequence of \(\{g_d(t)\}_0^\infty\) such that the degree is strictly increasing. Otherwise, the degrees of the \(g_d(t)\) are bounded by some \(N \geq 0\), so \(f(t)\) must itself be a polynomial of degree at most \(N\). Suppose \(f(t) = a\prod_{k=1}^n (t-t_k)\) for some \(\{t_k\}_1^n\) arranged in descending order, and let \(u := \max\{t_1, 1\}\). Set \(h_d(t) := \prod_{k=1}^d(1 - t/(d^2u+k\delta))\), and observe that \((1 - t/d^2)^d < \prod_{k=1}^d(1 - t/(d^2u+k\delta)) < (1 - t/(d^2u+d\delta))^d\) with the left and right sides converging locally uniformly to \(1\). Now let \(M\) be the smallest degree that occurs infinitely often in \(\{g_d(t)\}_0^\infty\) and \(\{d_k\}_1^\infty\) be the sequence of indices for which \(g_{d_k}(t)\) has degree \(M\). Then \(\{h_k(t)g_{d_k}(t)\}_0^\infty\) is a sequence of \(\delta\)-hyperbolic polynomials of strictly increasing degree converging locally uniformly to \(f(t)\).
\end{proof}

\noindent
In addition to the previous lemma, the following result will be useful in the proof of our main theorem \cite{cucker1989alternate}.

\begin{theorem}\label{continuity-of-roots}
Let \(p(x) = a_0 x^n + a_1 x^{n-1} + \dots + a_n\) be a polynomial in \(\C\) of degree \(n\) and let \(\alpha_1, \dots, \alpha_n\) be its roots. Then for all \(\varepsilon > 0\), there exists a \(\delta > 0\) such that for every polynomial \(q(x) = b_0 x^n + b_1 x^{n-1} + \dots + b_n\), if \(|b_i - a_i| < \delta\) for \(0 \leq i \leq n\), then there are \(\beta_1, \dots, \beta_n \in \mathbb{C}\) such that \(q(x) = \prod_{k=1}^n (x - \beta_i)\) and \(|\alpha_i - \beta_i| < \varepsilon\) for \(0 \leq i \leq n\).  
\end{theorem}

\section{Proof of Results}

With these observations, we are able to prove the following lemma, which describes the roots of \(A_d^{f,\delta}(t; x)\) as continuous curves in the \(t\)-\(x\) plane when \(f(t)\) is a \(\delta\)-hyperbolic polynomial. The lemma also states that the curves associated to successive \(\delta\)-Appell polynomials interlace each other, and it gives some of their limiting properties.

\begin{center}
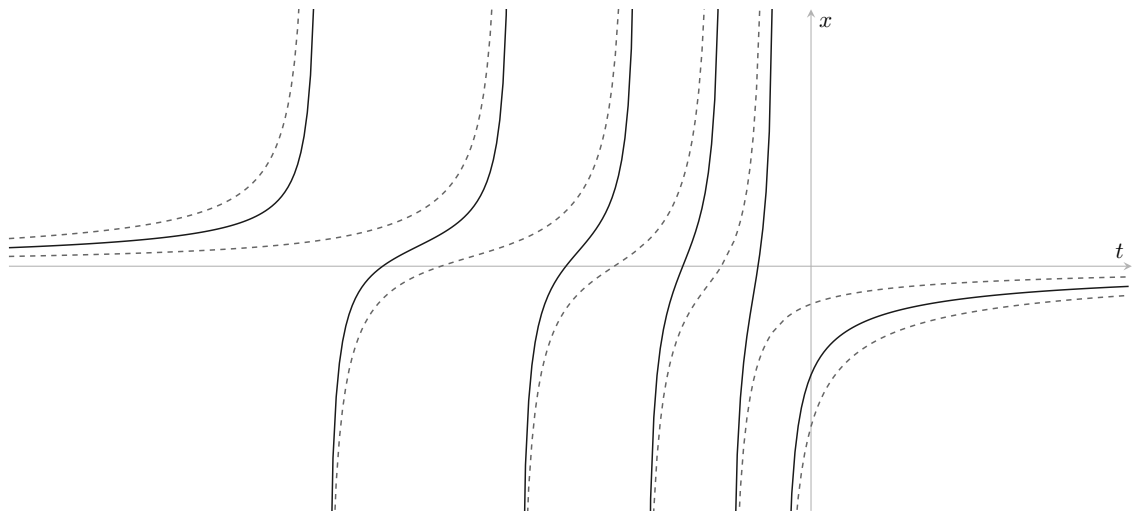
\begin{figure}[]
  \input{main-lemma}
  \caption{The zeros of the \(\frac{\pi^2}{6}\)-Appell polynomials for \(L_5^{3/2}(-\frac{x}{5})\) as given in \eqref{laguerre-polynomials}, which approximates the Bessel-Clifford function that occurs in \(R_1(n)\) in \eqref{rseries}. The zeros of the degree \(1\) polynomial are shown as solid lines, and the zeros of the degree \(2\) polynomial are shown as dashed lines.}
  \label{fig:main-lemma}
\end{figure}
\end{center}

\begin{lemma} \label{main-lemma}
Let \(f(t)\) be a \(\delta\)-hyperbolic polynomial of degree \(n\), and let \(\{t_k\}_1^n\) denote the roots of \(f(t)\) listed in descending order. Then for all \(0 \leq d \leq n\), there exist continuous functions \(\{\tau_{d,k}\}_1^{n+d}\) with \(\tau_{d,k} : (-\infty, 0) \to \R\) for \(1 \leq k \leq d\), \(\tau_{d,k} : \R \to \R\) for \(d+1 \leq k \leq n\), and \(\tau_{d,k} : (0, \infty) \to \R\) for \(n+1 \leq k \leq n+d\) such that the following statements hold: \begin{itemize}
    \item[\textnormal{(i)}] For all \(x, t \in \R\), we have \(A_d^{f,\delta}(t; x) = 0\) if and only if \(t = \tau_{d,k}(x)\) for some \(1 \leq k \leq n+d\).
    
    \smallskip
    \item[\textnormal{(ii)}] For all \(x\) and \(1 \leq k \leq n+d-1\), we have \(\tau_{d,k+1}(x) + \delta \leq \tau_{d-1,k}(x) \leq \tau_{d,k}(x)\).
    
    \smallskip
    \item[\textnormal{(iii)}] We have \begin{equation*}
        \lim_{x \to -\infty} \tau_{d,k}(x) = t_k \quad \text{and} \quad \lim_{x \to 0^-} \tau_{d,k}(x) = +\infty
    \end{equation*} for \(1 \leq k \leq d\), \begin{equation*}
        \lim_{x \to -\infty} \tau_{d,k}(x) = t_k \quad \text{and} \quad \lim_{x \to +\infty} \tau_{d,k}(x) = t_{k-d} - d\delta
    \end{equation*} for \(d+1 \leq k \leq n\), and \begin{equation*}
        \lim_{x \to 0^+} \tau_{d,k}(x) = -\infty \quad \text{and} \quad \lim_{x \to +\infty} \tau_{d,k}(x) = t_{k-d} - d\delta
    \end{equation*} for \(n+1 \leq k \leq n+d\).
\end{itemize}
\end{lemma}

\begin{proof}
By Theorem \ref{delta-difference-theorem}, we can construct functions \(\{\tau_{d,k}\}_1^{n+d}\) on the correct domains satisfying conditions (i) and (ii). Moreover, because \(A_d^{f,\delta}(t; x)\) is a continuously differentiable function in \(t\) and \(x\) for all \(d\) and \([D_t A_d^{f,\delta}(t; x)]_{t=\tau_{d,k}(x)}\) is nonzero for all \(x\), \(d\), and \(k\), the implicit function theorem implies that the functions \(\tau_{d,k}(x)\) are all continuous on their domains of definition.

Next, observe that the finite limits in (iii) follow from Theorem \ref{continuity-of-roots} since \(\lim_{x \to -\infty} A_d^{f,\delta}(t; x) = (-1)^d f(t)\) and \(\lim_{x \to +\infty} A_d^{f,\delta}(t; x) = f(t+d\delta)\). For the remaining limits, start by noting that if \(d = 0\), then there is nothing to check. Thus, suppose that \(d > 0\) and that the theorem is satisfied for \(d-1\). Then \(\lim_{x \to 0^-} \tau_{d,k}(x) \geq \lim_{x \to 0^-} \tau_{d-1,k}(x) = +\infty\) for \(1 \leq k \leq d-1\) and \(\lim_{x \to 0^+} \tau_{d,k}(x) \leq \lim_{x \to 0^+} \tau_{d-1,k-1}(x) = -\infty\) for \(n+2 \leq k \leq n+d\). We show that \(\lim_{x \to 0^-} \tau_{d,d}(x) = +\infty\). First, note that if \(L_s := \limsup_{x \to 0^-} \tau_{d,d}(x)\) and \(L_i := \liminf_{x \to 0^-} \tau_{d,d}(x)\), then \(L_s = L_i\); otherwise, for any \(L_s < t < L_i\), the Appell polynomial \(A_d^{f,\delta}(t; x)\) would have an infinite number of roots in \(x\). Moreover, if the \(L := \lim_{x \to 0^-} \tau_{d,d}(x)\) were finite, then by continuity of \(A_d^{f,\delta}(t; x)\), we could write \(A_d^{f,\delta}(L; 0) = 0\), implying that \(A_d^{f,\delta}(t; 0)\) has at least \(d-n+1\) roots. Since we know that \(A_d^{f,\delta}(t; 0)\) has exactly \(d-n\) roots, it follows that \(\lim_{x \to 0^-} \tau_{d,d}(x) = +\infty\). Similar reasoning shows that \(\lim_{x \to 0^+} \tau_{d,n+1}(x) = -\infty\), completing the induction.
\end{proof}

We immediately obtain our main theorem, which provides a simple criterion for hyperbolicity of the \(\delta\)-Appell polynomials.

\begin{proof}[Proof of Theorem~\ref{main-theorem}]
Using Lemma \ref{laguerre-polya-sequence-lemma}, there exists a sequence \(\{f_k(t)\}_0^\infty\) of \(\delta\)-hyperbolic polynomials with \(\deg f_{k+1}(t) > \deg f_k(t)\) for all \(k\) converging locally uniformly to \(f(t)\). Consequently, there exists an \(N \geq 0\) such that for all \(k \geq N\), we have \(\deg{f_k(t)} \geq d\). By uniform convergence, we can choose an \(N' \geq N\) such that \(f_k(t)\) has no zeros on \([t_0, t_0+d\delta]\) for \(k \geq N'\). Therefore, by Lemma~\ref{main-lemma}, the \(A_d^{f_k, \delta}(t_0; x)\) have \(d\) real, simple zeros for \(k \geq N'\). Because \(A_d^{f, \delta}(t_0; x) = \lim_{k \to \infty} A_d^{f_k, \delta}(t_0; x)\) coefficient-wise, Theorem \ref{continuity-of-roots} implies that \(A_d^{f,\delta}(t_0; x)\) has \(d\) real zeros.
\end{proof}

\noindent
In order to prove Theorem~\ref{new-LP-functions}, we show that functions in the \(\delta\)-\(\LP\) class do not grow ``too quickly.''

\begin{proof}[Proof of Theorem~\ref{new-LP-functions}]
By \eqref{laguerre-polya-function-form} we can write \begin{equation*}
    f(t_0+u) = ce^{\alpha u - \beta u^2}\prod_{n=1}^\infty \left(1 + \frac{u}{u_n}\right) e^{-\frac{u}{u_n}},
\end{equation*} where \(c \neq 0\) and \(\alpha \in \R\), \(\beta \geq 0\) and \(\{u_n\}_1^\infty\) is some (possibly finite) sequence of positive real numbers. In particular, this implies \begin{align*}
    |f(t_0+k\delta)| &\leq |c|e^{\alpha\delta k - \beta\delta^2 k^2},
\end{align*} so the function \begin{equation*}
    g(t) := \sum_{k=0}^\infty f(t_0+k\delta) \frac{x^k}{k!}
\end{equation*} is entire by the Cauchy-Hadamard theorem. The fact that \(g(t)\) lies in the \(\LP\) class is immediate from Theorem~\ref{main-theorem} using Jensen's theorem for the \(\LP\) class.
\end{proof}

\section{Applications}

In this section, we illustrate Theorem \ref{main-theorem} using several examples. To start, we use our theorem to show that the Jensen polynomials associated with Bessel-Clifford functions are hyperbolic.

\begin{proof}[Proof of Theorem~\ref{bessel-clifford-hyperbolicity}]
For \(\nu > -1\), it holds that all of the zeros of \(C_\nu(t)\) are simple and lie on the negative real axis. For \(\nu \geq 1/2\), we have that adjacent zeros of the Bessel function \(J_\nu(t)\) are separated by at least \(\pi\) \cite{segura2001bounds}. Consequently, setting \(r_k^\nu\) to be the absolute value of the \(k\)th zero of \(C_\nu(t)\) from the right for \(\nu > -1\), we find that \begin{equation*}
    C_\nu(t) = \frac{1}{\Gamma(\nu+1)} \prod_{n=1}^\infty \left(1 + \frac{t}{r_n^\nu}\right)
\end{equation*} with \(r_{n+1}^\nu - r_n^\nu > \pi^2/4\) for all \(n\). Thus, by Theorem~\ref{main-theorem}, the polynomials \(A_d^{C_\nu,\delta}(t; x)\) are hyperbolic for \(0 < \delta \leq \pi^2/4\) and \(t \geq -r_1^\nu\).
\end{proof}

This settles Ono's conjecture that the Jensen polynomials of the first term of the Hardy-Ramanujan-Rademacher expansion for the partition function are always hyperbolic.

\begin{proof}[Proof of Corollary~\ref{ono-conjecture-corollary}]
By Theorem \ref{bessel-clifford-hyperbolicity}, the Appell polynomials \begin{equation*}
    A_d^{R_\alpha, 1}(n; x) = \sum_{k=0}^d \binom{d}{k} (-1)^{d-k} R_\alpha(n+k) e^{kx}
\end{equation*} have \(d\) real roots whenever \(n \geq \alpha/24\). Moreover, using the relation \begin{equation*}
    e^{dx}J_d^{R_\alpha, n}(-e^{-x}) = A_d^{R_\alpha, 1}(n; x),
\end{equation*} we see that each root of \(A_d^{R_\alpha, 1}(n; x)\) corresponds to a unique root of \(J_d^{R_\alpha, n}(x)\). Consequently, the Jensen polynomials \(J_d^{R_\alpha, n}(x)\) are hyperbolic.
\end{proof}

\noindent
The Bessel-Clifford functions represent a typical case of the application of our theorem: in the positive direction they grow exponentially, whereas in the negative direction they oscillate between positive and negative values, with the separations between zeros bounded below by \(\pi^2/4\).

We also provide two exceptional applications. The first will be the family of functions of the form \(e^{-\beta t^2}\) with \(\beta \geq 0\); i.e., the Gaussian distributions. These functions are noteworthy because they belong to every \(\delta\)-\(\LP\) class. Two immediate consequences are that the polynomials \(g_d(x) := \sum_{k=0}^d (-1)^{d-k} \binom{d}{k} e^{-\beta k^2+kx}\) all have \(d\) real roots, and that the function \begin{equation}
    g(x) := \sum_{k=0}^\infty \frac{(-1)^k}{k!} e^{-\beta k^2+kx}
\end{equation} is such that \(g(\log{x})\) belongs to the \(\LP\) class. In fact, using the recursion formula \begin{equation*}
    g_{d+1}(x) = g_d(x) - e^{x-\beta}g_d(x-2\beta)
\end{equation*} and reasoning similar to that used in the proof of Lemma \ref{main-lemma}, one finds that the roots of \(g_d(x)\) are separated by \(2\beta\). From here, the limit formula \begin{equation*}
    g(x) = \lim_{d \to \infty} (-1)^dg_d(x-\log{d})
\end{equation*} yields the same result for \(g(x)\), although \(g(x)\) is not in the \(2\beta\)-\(\LP\) class.

The second example we consider is the reciprocal Gamma function, which satisfies the identity \begin{equation} \label{gamma-product}
    \frac{e^{-\gamma z}}{\Gamma(z)} = z\prod_{n=1}^\infty \left(1 + \frac{z}{n}\right) e^{-\frac{z}{n}}
\end{equation} from \cite{nist}, where \(\gamma\) denotes the Euler-Mascheroni constant. As a consequence, \(1/\Gamma(t)\) is in the \(1\)-\(\LP\) class. This function is significant not only because it represents an ``extreme'' element of the \(1\)-\(\LP\) class, but also because it is essential for the definitions of numerous families of orthogonal polynomials. For example, the associated Laguerre polynomials are defined by \begin{equation} \label{laguerre-polynomials}
    L_d^\nu(-x) := \sum_{k=0}^d \binom{d+\nu}{d-k} \frac{x^k}{k!} = \frac{\Gamma(d+\nu+1)}{d!} \sum_{k=0}^d \binom{d}{k} \frac{1}{\Gamma(\nu+k+1)!} x^k,
\end{equation} and can be normalized so that they coincide with the ordinary Appell polynomials \(A_d^{C_\nu}(0; x)\). An alternative definition that is more in line with the theme of this paper associates the Laguerre polynomials with the Jensen polynomials of the reciprocal Gamma function: \begin{equation*}
    L_d^\nu(-x) = \frac{\Gamma(d+\nu+1)}{d!} J_d^{\frac{1}{\Gamma},\nu+1}(x).
\end{equation*} This allows us to obtain a \(\delta\)-generalization of the Laguerre polynomials; namely, the Jensen polynomials \(J_d^{\gamma_\delta,\nu+1}(x)\) for the functions \(\gamma_\delta(t) := 1/\Gamma(\delta t)\). Our main theorem automatically implies that these Jensen polynomials remain hyperbolic for \(0 < \delta < 1\) in addition to being orthogonal at \(\delta = 1\).

\bibliographystyle{plain}
\bibliography{references}

\Addresses

\end{document}

%% file: delta-difference-theorem.tex
\pgfplotsset{
    standard/.style={
        every axis x label/.style={at={(current axis.south east)}, anchor=south},
        every axis y label/.style={at={(current axis.north west)},anchor=west}, 
        every axis label/.append style={scale=0.8}
    }
}

\begin{tikzpicture}

\begin{axis}[
    standard,
    axis y line = middle,
    axis x line = middle,
    axis line style={black!30},
    x grid style={white!69.0196078431373!black},
    xmin=-24, xmax=1,
    y grid style={white!69.0196078431373!black},
    ymin=-4, ymax=4,
    xtick={0}, 
    ytick={0},
    width=1.0\linewidth,
    height=0.35\linewidth
]
\addplot [black!30]
table {%
-1.99897772310717 0
-1.99897772310717 1.1
};
\addplot [black!30]
table {%
-8.40232516245077 -0.45
-8.40232516245077 0
};
\addplot [black!30]
table {%
-21.0986971144421 0
-21.0986971144421 1.3
};
\addplot [black!30]
table {%
-3.64391178995534 0
-3.64391178995534 0.85
};
\addplot [black!30]
table {%
-10.047259229299 -0.9
-10.047259229299 0
};
\addplot [black!30]
table {%
-22.7436311812903 0
-22.7436311812903 3.0
};

\addplot [black!30, only marks, mark options={scale=0.5}] table {
-1.99897772310717 1.1
-8.40232516245077 -0.45
-21.0986971144421 1.3
-3.64391178995534 0.85
-10.047259229299 -0.9
-22.7436311812903 3.0
};
\addplot [line width=0.2mm, black!60, dash pattern=on 2pt off 2pt]
table {%
-28 21.7075617283951
-27.9 21.202
-27.8 20.7028827160494
-27.7 20.2101728395062
-27.6 19.7238333333334
-27.5 19.2438271604938
-27.4 18.7701172839506
-27.3 18.3026666666667
-27.2 17.8414382716049
-27.1 17.3863950617284
-27 16.9375
-26.9 16.4947160493827
-26.8 16.0580061728395
-26.7 15.6273333333333
-26.6 15.2026604938272
-26.5 14.783950617284
-26.4 14.3711666666666
-26.3 13.9642716049383
-26.2 13.5632283950617
-26.1 13.168
-26 12.778549382716
-25.9 12.3948395061728
-25.8 12.0168333333333
-25.7 11.6444938271605
-25.6 11.2777839506173
-25.5 10.9166666666667
-25.4 10.5611049382716
-25.3 10.2110617283951
-25.2 9.86650000000001
-25.1 9.52738271604938
-25 9.19367283950617
-24.9 8.86533333333331
-24.8 8.54232716049383
-24.7 8.22461728395061
-24.6 7.91216666666667
-24.5 7.60493827160493
-24.4 7.30289506172839
-24.3 7.00599999999999
-24.2 6.71421604938271
-24.1 6.42750617283951
-24 6.14583333333333
-23.9 5.86916049382715
-23.8 5.59745061728395
-23.7 5.33066666666666
-23.6 5.06877160493827
-23.5 4.81172839506173
-23.4 4.5595
-23.3 4.31204938271605
-23.2 4.06933950617284
-23.1 3.83133333333334
-23 3.5979938271605
-22.9 3.36928395061728
-22.8 3.14516666666667
-22.7 2.9256049382716
-22.6 2.71056172839507
-22.5 2.5
-22.4 2.29388271604938
-22.3 2.09217283950617
-22.2 1.89483333333333
-22.1 1.70182716049383
-22 1.51311728395062
-21.9 1.32866666666667
-21.8 1.14843827160494
-21.7 0.972395061728401
-21.6 0.800500000000003
-21.5 0.632716049382719
-21.4 0.469006172839505
-21.3 0.309333333333338
-21.2 0.153660493827162
-21.1 0.00195061728395507
-21 -0.145833333333333
-20.9 -0.289728395061734
-20.8 -0.429771604938269
-20.7 -0.566000000000003
-20.6 -0.698450617283948
-20.5 -0.827160493827159
-20.4 -0.952166666666668
-20.3 -1.0735061728395
-20.2 -1.19121604938271
-20.1 -1.30533333333333
-20 -1.41589506172839
-19.9 -1.52293827160494
-19.8 -1.6265
-19.7 -1.72661728395062
-19.6 -1.82332716049383
-19.5 -1.91666666666667
-19.4 -2.00667283950617
-19.3 -2.09338271604938
-19.2 -2.17683333333333
-19.1 -2.25706172839506
-19 -2.3341049382716
-18.9 -2.408
-18.8 -2.47878395061728
-18.7 -2.54649382716049
-18.6 -2.61116666666667
-18.5 -2.67283950617284
-18.4 -2.73154938271605
-18.3 -2.78733333333333
-18.2 -2.84022839506173
-18.1 -2.89027160493827
-18 -2.9375
-17.9 -2.98195061728395
-17.8 -3.02366049382716
-17.7 -3.06266666666667
-17.6 -3.09900617283951
-17.5 -3.13271604938272
-17.4 -3.16383333333333
-17.3 -3.19239506172839
-17.2 -3.21843827160494
-17.1 -3.242
-17 -3.26311728395062
-16.9 -3.28182716049383
-16.8 -3.29816666666667
-16.7 -3.31217283950617
-16.6 -3.32388271604938
-16.5 -3.33333333333333
-16.4 -3.34056172839506
-16.3 -3.3456049382716
-16.2 -3.3485
-16.1 -3.34928395061728
-16 -3.34799382716049
-15.9 -3.34466666666667
-15.8 -3.33933950617284
-15.7 -3.33204938271605
-15.6 -3.32283333333333
-15.5 -3.31172839506173
-15.4 -3.29877160493827
-15.3 -3.284
-15.2 -3.26745061728395
-15.1 -3.24916049382716
-15 -3.22916666666667
-14.9 -3.20750617283951
-14.8 -3.18421604938272
-14.7 -3.15933333333333
-14.6 -3.13289506172839
-14.5 -3.10493827160494
-14.4 -3.0755
-14.3 -3.04461728395062
-14.2 -3.01232716049383
-14.1 -2.97866666666667
-14 -2.94367283950617
-13.9 -2.90738271604938
-13.8 -2.86983333333333
-13.7 -2.83106172839506
-13.6 -2.7911049382716
-13.5 -2.75
-13.4 -2.70778395061728
-13.3 -2.66449382716049
-13.2 -2.62016666666667
-13.1 -2.57483950617284
-13 -2.52854938271605
-12.9 -2.48133333333333
-12.8 -2.43322839506173
-12.7 -2.38427160493827
-12.6 -2.3345
-12.5 -2.28395061728395
-12.4 -2.23266049382716
-12.3 -2.18066666666667
-12.2 -2.1280061728395
-12.1 -2.07471604938272
-12 -2.02083333333333
-11.9 -1.96639506172839
-11.8 -1.91143827160494
-11.7 -1.856
-11.6 -1.80011728395062
-11.5 -1.74382716049383
-11.4 -1.68716666666667
-11.3 -1.63017283950617
-11.2 -1.57288271604938
-11.1 -1.51533333333333
-11 -1.45756172839506
-10.9 -1.3996049382716
-10.8 -1.3415
-10.7 -1.28328395061728
-10.6 -1.22499382716049
-10.5 -1.16666666666667
-10.4 -1.10833950617284
-10.3 -1.05004938271605
-10.2 -0.991833333333333
-10.1 -0.933728395061728
-10 -0.875771604938272
-9.9 -0.818
-9.8 -0.760450617283951
-9.7 -0.70316049382716
-9.6 -0.646166666666666
-9.5 -0.589506172839506
-9.4 -0.533216049382716
-9.3 -0.477333333333333
-9.2 -0.421895061728395
-9.1 -0.366938271604938
-9 -0.3125
-8.9 -0.258617283950617
-8.8 -0.205327160493827
-8.7 -0.152666666666666
-8.6 -0.100672839506173
-8.5 -0.0493827160493826
-8.4 0.0011666666666663
-8.3 0.0509382716049376
-8.2 0.0998950617283955
-8.1 0.148000000000001
-8 0.195216049382716
-7.9 0.241506172839507
-7.8 0.286833333333333
-7.7 0.33116049382716
-7.6 0.374450617283951
-7.5 0.416666666666667
-7.4 0.457771604938272
-7.3 0.497728395061729
-7.2 0.5365
-7.1 0.57404938271605
-7 0.610339506172839
-6.9 0.645333333333333
-6.8 0.678993827160494
-6.7 0.711283950617284
-6.6 0.742166666666667
-6.5 0.771604938271605
-6.4 0.799561728395062
-6.3 0.826
-6.2 0.850882716049383
-6.1 0.874172839506173
-6 0.895833333333333
-5.9 0.915827160493827
-5.8 0.934117283950617
-5.7 0.950666666666667
-5.6 0.965438271604938
-5.5 0.978395061728395
-5.4 0.9895
-5.3 0.998716049382716
-5.2 1.00600617283951
-5.1 1.01133333333333
-5 1.01466049382716
-4.9 1.01595061728395
-4.8 1.01516666666667
-4.7 1.01227160493827
-4.6 1.00722839506173
-4.5 1
-4.4 0.99054938271605
-4.3 0.97883950617284
-4.2 0.964833333333333
-4.1 0.948493827160494
-4 0.929783950617284
-3.9 0.908666666666667
-3.8 0.885104938271605
-3.7 0.859061728395062
-3.6 0.8305
-3.5 0.799382716049383
-3.4 0.765672839506173
-3.3 0.729333333333333
-3.2 0.690327160493827
-3.1 0.648617283950617
-3 0.604166666666667
-2.9 0.556938271604938
-2.8 0.506895061728395
-2.7 0.454
-2.6 0.398216049382716
-2.5 0.339506172839506
-2.4 0.277833333333333
-2.3 0.21316049382716
-2.2 0.145450617283951
-2.1 0.0746666666666669
-2 0.000771604938271629
-1.9 -0.0762716049382716
-1.8 -0.1565
-1.7 -0.239950617283951
-1.6 -0.326660493827161
-1.5 -0.416666666666667
-1.4 -0.510006172839506
-1.3 -0.606716049382716
-1.2 -0.706833333333333
-1.1 -0.810395061728395
-1 -0.917438271604938
-0.9 -1.028
-0.8 -1.14211728395062
-0.7 -1.25982716049383
-0.6 -1.38116666666667
-0.5 -1.50617283950617
-0.4 -1.63488271604938
-0.3 -1.76733333333333
-0.2 -1.90356172839506
-0.1 -2.04360493827161
0 -2.1875
0.1 -2.33528395061728
0.2 -2.48699382716049
0.3 -2.64266666666667
0.4 -2.80233950617284
0.5 -2.96604938271605
0.6 -3.13383333333333
0.7 -3.30572839506173
0.8 -3.48177160493827
0.9 -3.662
1 -3.84645061728395
1.1 -4.03516049382716
1.2 -4.22816666666667
1.3 -4.42550617283951
1.4 -4.62721604938272
1.5 -4.83333333333333
1.6 -5.0438950617284
1.7 -5.25893827160494
1.8 -5.4785
1.9 -5.70261728395062
2 -5.93132716049383
2.1 -6.16466666666667
2.2 -6.40267283950617
2.3 -6.64538271604938
2.4 -6.89283333333333
2.5 -7.14506172839506
2.6 -7.40210493827161
2.7 -7.664
2.8 -7.93078395061728
2.9 -8.20249382716049
3 -8.47916666666667
3.1 -8.76083950617284
3.2 -9.04754938271605
3.3 -9.33933333333333
3.4 -9.63622839506173
3.5 -9.93827160493827
3.6 -10.2455
3.7 -10.557950617284
3.8 -10.8756604938272
3.9 -11.1986666666667
4 -11.5270061728395
4.1 -11.8607160493827
4.2 -12.1998333333333
4.3 -12.5443950617284
4.4 -12.8944382716049
4.5 -13.25
4.6 -13.6111172839506
4.7 -13.9778271604938
4.8 -14.3501666666667
4.9 -14.7281728395062
};
\addplot [line width=0.2mm, black!90]
table {%
-28 -9.40400100600276
-27.9 -9.13604382278325
-27.8 -8.87195440398185
-27.7 -8.61170820026033
-27.6 -8.35528066228052
-27.5 -8.10264724070422
-27.4 -7.85378338619324
-27.3 -7.60866454940941
-27.2 -7.36726618101453
-27.1 -7.12956373167038
-27 -6.89553265203881
-26.9 -6.66514839278161
-26.8 -6.43838640456059
-26.7 -6.21522213803756
-26.6 -5.99563104387433
-26.5 -5.7795885727327
-26.4 -5.56707017527449
-26.3 -5.35805130216153
-26.2 -5.15250740405558
-26.1 -4.95041393161849
-26 -4.75174633551205
-25.9 -4.55648006639808
-25.8 -4.36459057493839
-25.7 -4.17605331179478
-25.6 -3.99084372762907
-25.5 -3.80893727310305
-25.4 -3.63030939887854
-25.3 -3.45493555561737
-25.2 -3.28279119398131
-25.1 -3.11385176463222
-25 -2.94809271823185
-24.9 -2.78548950544205
-24.8 -2.62601757692462
-24.7 -2.46965238334137
-24.6 -2.3163693753541
-24.5 -2.16614400362463
-24.4 -2.01895171881477
-24.3 -1.87476797158632
-24.2 -1.7335682126011
-24.1 -1.59532789252091
-24 -1.46002246200757
-23.9 -1.32762737172287
-23.8 -1.19811807232864
-23.7 -1.07147001448668
-23.6 -0.947658648858805
-23.5 -0.826659426106811
-23.4 -0.708447796892525
-23.3 -0.592999211877744
-23.2 -0.480289121724283
-23.1 -0.370292977093949
-23 -0.262986228648553
-22.9 -0.158344327049899
-22.8 -0.0563427229598083
-22.7 0.0430431329599263
-22.6 0.139837790047479
-22.5 0.234065797641051
-22.4 0.325751705078831
-22.3 0.414920061699004
-22.2 0.501595416839771
-22.1 0.585802319839311
-22 0.667565320035826
-21.9 0.746908966767498
-21.8 0.823857809372521
-21.7 0.898436397189085
-21.6 0.97066927955538
-21.5 1.0405810058096
-21.4 1.10819612528993
-21.3 1.17353918733457
-21.2 1.2366347412817
-21.1 1.29750733646951
-21 1.35618152223621
-20.9 1.41268184791997
-20.8 1.46703286285898
-20.7 1.51925911639144
-20.6 1.56938515785554
-20.5 1.61743553658947
-20.4 1.66343480193142
-20.3 1.70740750321958
-20.2 1.74937818979214
-20.1 1.7893714109873
-20 1.82741171614323
-19.9 1.86352365459814
-19.8 1.89773177569021
-19.7 1.93006062875764
-19.6 1.96053476313861
-19.5 1.98917872817131
-19.4 2.01601707319394
-19.3 2.04107434754469
-19.2 2.06437510056175
-19.1 2.0859438815833
-19 2.10580523994754
-18.9 2.12398372499266
-18.8 2.14050388605685
-18.7 2.1553902724783
-18.6 2.16866743359521
-18.5 2.18035991874575
-18.4 2.19049227726813
-18.3 2.19908905850053
-18.2 2.20617481178114
-18.1 2.21177408644816
-18 2.21591143183977
-17.9 2.21861139729417
-17.8 2.21989853214955
-17.7 2.21979738574409
-17.6 2.21833250741599
-17.5 2.21552844650344
-17.4 2.21140975234462
-17.3 2.20600097427774
-17.2 2.19932666164097
-17.1 2.19141136377252
-17 2.18227963001057
-16.9 2.17195600969331
-16.8 2.16046505215893
-16.7 2.14783130674563
-16.6 2.13407932279159
-16.5 2.119233649635
-16.4 2.10331883661406
-16.3 2.08635943306696
-16.2 2.06837998833188
-16.1 2.04940505174702
-16 2.02945917265057
-15.9 2.00856690038071
-15.8 1.98675278427565
-15.7 1.96404137367356
-15.6 1.94045721791265
-15.5 1.91602486633109
-15.4 1.89076886826709
-15.3 1.86471377305883
-15.2 1.8378841300445
-15.1 1.8103044885623
-15 1.78199939795041
-14.9 1.75299340754702
-14.8 1.72331106669034
-14.7 1.69297692471853
-14.6 1.66201553096981
-14.5 1.63045143478235
-14.4 1.59830918549434
-14.3 1.56561333244399
-14.2 1.53238842496948
-14.1 1.498659012409
-14 1.46444964410073
-13.9 1.42978486938288
-13.8 1.39468923759363
-13.7 1.35918729807118
-13.6 1.3233036001537
-13.5 1.2870626931794
-13.4 1.25048912648647
-13.3 1.21360744941309
-13.2 1.17644221129745
-13.1 1.13901796147775
-13 1.10135924929218
-12.9 1.06349062407893
-12.8 1.02543663517618
-12.7 0.987221831922134
-12.6 0.948870763654977
-12.5 0.9104079797129
-12.4 0.871858029434094
-12.3 0.833245462156749
-12.2 0.794594827219056
-12.1 0.755930673959206
-12 0.717277551715389
-11.9 0.678660009825797
-11.8 0.640102597628619
-11.7 0.601629864462046
-11.6 0.56326635966427
-11.5 0.52503663257348
-11.4 0.486965232527867
-11.3 0.449076708865622
-11.2 0.411395610924935
-11.1 0.373946488043997
-11 0.336753889561
-10.9 0.299842364814133
-10.8 0.263236463141587
-10.7 0.226960733881552
-10.6 0.19103972637222
-10.5 0.155497989951781
-10.4 0.120360073958426
-10.3 0.0856505277303449
-10.2 0.0513939006057279
-10.1 0.0176147419227673
-10 -0.0156623989803472
-9.9 -0.0484129727654249
-9.8 -0.0806124300942751
-9.7 -0.112236221628708
-9.6 -0.143259798030532
-9.5 -0.173658609961555
-9.4 -0.20340810808359
-9.3 -0.232483743058443
-9.2 -0.260860965547926
-9.1 -0.288515226213846
-9 -0.315421975718014
-8.9 -0.341556664722238
-8.8 -0.366894743888329
-8.7 -0.391411663878095
-8.6 -0.415082875353346
-8.5 -0.437883828975891
-8.4 -0.45978997540754
-8.3 -0.480776765310102
-8.2 -0.500819649345386
-8.1 -0.519894078175202
-8 -0.537975502461358
-7.9 -0.555039372865666
-7.8 -0.571061140049933
-7.7 -0.586016254675969
-7.6 -0.599880167405583
-7.5 -0.612628328900586
-7.4 -0.624236189822785
-7.3 -0.634679200833991
-7.2 -0.643932812596013
-7.1 -0.651972475770661
-7 -0.658773641019742
-6.9 -0.664311759005068
-6.8 -0.668562280388448
-6.7 -0.67150065583169
-6.6 -0.673102335996604
-6.5 -0.673342771544999
-6.4 -0.672197413138686
-6.3 -0.669641711439472
-6.2 -0.665651117109168
-6.1 -0.660201080809583
-6 -0.653267053202526
-5.9 -0.644824484949807
-5.8 -0.634848826713234
-5.7 -0.623315529154618
-5.6 -0.610200042935768
-5.5 -0.595477818718492
-5.4 -0.579124307164601
-5.3 -0.561114958935904
-5.2 -0.54142522469421
-5.1 -0.520030555101328
-5 -0.496906400819069
-4.9 -0.47202821250924
-4.8 -0.445371440833652
-4.7 -0.416911536454114
-4.6 -0.386623950032435
-4.5 -0.354484132230425
-4.4 -0.320467533709893
-4.3 -0.284549605132648
-4.2 -0.2467057971605
-4.1 -0.206911560455258
-4 -0.165142345678732
-3.9 -0.12137360349273
-3.8 -0.0755807845590622
-3.7 -0.0277393395395381
-3.6 0.022175280904033
-3.5 0.0741876261098421
-3.4 0.128322245416079
-3.3 0.184603688160936
-3.2 0.243056503682602
-3.1 0.303705241319268
-3 0.366574450409126
-2.9 0.431688680290365
-2.8 0.499072480301176
-2.7 0.56875039977975
-2.6 0.640746988064277
-2.5 0.715086794492949
-2.4 0.791794368403956
-2.3 0.870894259135488
-2.2 0.952411016025736
-2.1 1.03636918841289
-2 1.12279332563514
-1.9 1.21170797703068
-1.8 1.3031376919377
-1.7 1.39710701969439
-1.6 1.49364050963894
-1.5 1.59276271110954
-1.4 1.69449817344438
-1.3 1.79887144598165
-1.2 1.90590707805954
-1.1 2.01562961901625
-1 2.12806361818996
-0.9 2.24323362491886
-0.8 2.36116418854115
-0.7 2.48187985839502
-0.6 2.60540518381865
-0.5 2.73176471415024
-0.4 2.86098299872798
-0.3 2.99308458689005
-0.2 3.12809402797466
-0.1 3.26603587131998
0 3.40693466626421
0.1 3.55081496214555
0.2 3.69770130830217
0.3 3.84761825407228
0.4 4.00059034879406
0.5 4.15664214180571
0.6 4.3157981824454
0.7 4.47808302005135
0.8 4.64352120396172
0.9 4.81213728351473
1 4.98395580804855
1.1 5.15900132690138
1.2 5.3372983894114
1.3 5.51887154491681
1.4 5.70374534275581
1.5 5.89194433226657
1.6 6.08349306278729
1.7 6.27841608365617
1.8 6.47673794421138
1.9 6.67848319379113
2 6.8836763817336
2.1 7.09234205737699
2.2 7.30450477005948
2.3 7.52018906911926
2.4 7.73941950389453
2.5 7.96222062372347
2.6 8.18861697794428
2.7 8.41863311589515
2.8 8.65229358691427
2.9 8.88962294033983
3 9.13064572551001
3.1 9.37538649176302
3.2 9.62386978843704
3.3 9.87612016487025
3.4 10.1321621704009
3.5 10.3920203543671
3.6 10.655719266107
3.7 10.9232834549589
3.8 11.194737470261
3.9 11.4701058613515
4 11.7494131775684
4.1 12.0326839682501
4.2 12.3199427827347
4.3 12.6112141703604
4.4 12.9065226804655
4.5 13.205892862388
4.6 13.5093492654661
4.7 13.8169164390382
4.8 14.1286189324423
4.9 14.4444812950167
};

\end{axis}

\end{tikzpicture}

%% file: main-lemma.tex
\pgfplotsset{
    standard/.style={
        every axis x label/.style={at={(current axis.south east)}, anchor=south},
        every axis y label/.style={at={(current axis.north west)},anchor=west}, 
        every axis label/.append style={scale=0.8}
    }
}
\begin{tikzpicture}
\definecolor{color0}{rgb}{0.474509803921569,0.447058823529412,0.43921568627451}

\begin{axis}[
 standard,
 axis y line = middle,
 axis x line = middle,
 axis line style={black!30},
 xlabel=$t$, ylabel=$x$,
 xmin=-25, xmax=10,
 ymin=-4, ymax=4.2,
 xtick={0}, 
 ytick={0},
 width=1.0\linewidth,
 height=0.5\linewidth
]
\addplot [line width=0.2mm, black!90]
table {%
-25 0.302500247383071
-24.9 0.304591192715638
-24.8 0.306715226159217
-24.7 0.308873227443303
-24.6 0.311066110151931
-24.5 0.313294823444108
-24.4 0.315560353882979
-24.3 0.317863727382087
-24.2 0.320206011277575
-24.1 0.322588316536173
-24 0.325011800109537
-23.9 0.32747766744663
-23.8 0.329987175176745
-23.7 0.332541633977107
-23.6 0.335142411640158
-23.5 0.337790936357274
-23.4 0.340488700237009
-23.3 0.343237263078074
-23.2 0.346038256418925
-23.1 0.348893387888372
-23 0.351804445883761
-22.9 0.354773304606457
-22.8 0.357801929486961
-22.7 0.360892383035942
-22.6 0.364046831160976
-22.5 0.36726754999337
-22.4 0.37055693327425
-22.3 0.373917500354701
-22.2 0.377351904870869
-22.1 0.38086294416218
-22 0.384453569508624
-21.9 0.388126897272253
-21.8 0.391886221038324
-21.7 0.395735024863177
-21.6 0.399676997748456
-21.5 0.403716049483458
-21.4 0.407856327993996
-21.3 0.412102238391278
-21.2 0.416458463902327
-21.1 0.420929988911546
-21 0.425522124366372
-20.9 0.430240535836586
-20.8 0.435091274559314
-20.7 0.440080811848797
-20.6 0.445216077307258
-20.5 0.450504501338606
-20.4 0.45595406254457
-20.3 0.461573340673642
-20.2 0.467371575901236
-20.1 0.473358735346505
-20 0.47954558788323
-19.9 0.485943788482481
-19.8 0.492565973541257
-19.7 0.499425868910776
-19.6 0.506538412651355
-19.5 0.513919894919815
-19.4 0.521588117856111
-19.3 0.529562578898606
-19.2 0.537864681647404
-19.1 0.546517979245812
-19 0.555548456303327
-18.9 0.564984856695675
-18.8 0.574859066219694
-18.7 0.585206561149906
-18.6 0.59606693636484
-18.5 0.607484530054748
-18.4 0.619509166315915
-18.3 0.632197042490458
-18.2 0.645611795349709
-18.1 0.659825789732542
-18 0.674921685861076
-17.9 0.690994358428566
-17.8 0.708153263353761
-17.7 0.726525379231378
-17.6 0.746258893516374
-17.5 0.767527863620516
-17.4 0.790538168321597
-17.3 0.815535187398536
-17.2 0.842813826306152
-17.1 0.872731768441587
-17 0.905727239799341
-16.9 0.94234319231662
-16.8 0.983260794953353
-16.7 1.02934671560157
-16.6 1.08172133814816
-16.5 1.1418596489194
-16.4 1.21174474350125
-16.3 1.29410926136026
-16.2 1.39283023890216
-16.1 1.51360586201952
-16 1.66518397300055
-15.9 1.86175953789951
-15.8 2.12811717517208
-15.7 2.51219328507516
-15.6 3.12232507861593
-15.5 4.28159839188264
-15.4 8.07393865695511
};
\addplot [line width=0.2mm, black!90]
table {%
-15 -5.83342253179356
-14.9 -3.13785479398368
-14.8 -2.13807905831057
-14.7 -1.58582939418229
-14.6 -1.22929515790248
-14.5 -0.977763078892
-14.4 -0.789575004851502
-14.3 -0.642679762183337
-14.2 -0.524236940465014
-14.1 -0.426233514618703
-14 -0.343397009742867
-13.9 -0.272109129043594
-13.8 -0.209799805342916
-13.7 -0.154590056646963
-13.6 -0.105071740874528
-13.5 -0.0601664827242437
-13.4 -0.0190323387037802
-13.3 0.0189997137263065
-13.2 0.0544700775822745
-13.1 0.0878221342208141
-13 0.119425245555997
-12.9 0.14959182426197
-12.8 0.178590241313614
-12.7 0.206654772945858
-12.6 0.233993421695653
-12.5 0.260794204092212
-12.4 0.287230336014626
-12.3 0.313464638319035
-12.2 0.339653413013026
-12.1 0.365949993399749
-12 0.392508143719727
-11.9 0.419485471166507
-11.8 0.447047014108494
-11.7 0.475369185004682
-11.6 0.504644276825391
-11.5 0.535085791753748
-11.4 0.566934927611809
-11.3 0.600468672138646
-11.2 0.636010126167613
-11.1 0.673941933248586
-11 0.714724083360094
-10.9 0.758917962271693
-10.8 0.807219473184385
-10.7 0.860505605373996
-10.6 0.919901405576457
-10.5 0.986878751507966
-10.4 1.06340626444363
-10.3 1.15218449378082
-10.2 1.2570294979914
-10.1 1.38352822394098
-10 1.540223752185
-9.9 1.7409174340944
-9.8 2.00957619086159
-9.7 2.39221482638765
-9.6 2.99164771000942
-9.5 4.10777118891036
-9.4 7.50909523396341
};
\addplot [line width=0.2mm, black!90]
table {%
-9 -6.20451296704453
-8.9 -3.25085302464962
-8.8 -2.19937838802054
-8.7 -1.62150071400085
-8.6 -1.24706143566991
-8.5 -0.980595048755488
-8.4 -0.778681275615809
-8.3 -0.618455164162038
-8.2 -0.486630700430507
-8.1 -0.374910284768689
-8 -0.277811154409881
-7.9 -0.19153767022361
-7.8 -0.113353041523259
-7.7 -0.0412082956476644
-7.6 0.0264881652167014
-7.5 0.0910205691666502
-7.4 0.153468238066042
-7.3 0.214779536430712
-7.2 0.275829011223748
-7.1 0.337465534134816
-7 0.40055732130573
-6.9 0.466039058021794
-6.8 0.534966785212435
-6.7 0.608587850697301
-6.6 0.688436650906515
-6.5 0.77647335408557
-6.4 0.875294987501482
-6.3 0.988472015947861
-6.2 1.12111219161037
-6.1 1.28085993079419
-6 1.47979266333652
-5.9 1.73834513990246
-5.8 2.09443160967999
-5.7 2.62851691021405
-5.6 3.55401801599085
-5.5 5.7966580769178
};
\addplot [line width=0.2mm, black!90]
table {%
-5.1 -9.20718339272454
-5 -3.80172674462499
-4.9 -2.47275908596424
-4.8 -1.78506239277099
-4.7 -1.34678197289039
-4.6 -1.03382536763249
-4.5 -0.792591416656648
-4.4 -0.595565235997878
-4.3 -0.42684730207179
-4.2 -0.276335138931131
-4.1 -0.1370180306261
-4 -0.00355803820516253
-3.9 0.128555222815529
-3.8 0.263600262521285
-3.7 0.406152336911622
-3.6 0.561692648751556
-3.5 0.737487451973205
-3.4 0.944081085086838
-3.3 1.19816946106512
-3.2 1.52884411202294
-3.1 1.9933403833943
-3 2.72645280780784
-2.9 4.17032369484461
-2.8 11.6439164025775
};
\addplot [line width=0.2mm, black!90]
table {%
-2.4 -5.15090068095769
-2.3 -3.07911842611223
-2.2 -2.15050654835431
-2.1 -1.57347959889252
-2 -1.15035543563102
-1.9 -0.801123678380778
-1.8 -0.482420713546208
-1.7 -0.162763776296006
-1.6 0.189875889786302
-1.5 0.621007416569804
-1.4 1.21694204504082
-1.3 2.19835358000341
-1.2 4.52376537896732
};
\addplot [line width=0.2mm, black!90]
table {%
-0.8 -10.1848504376792
-0.7 -5.02952525799597
-0.6 -3.71784707952137
-0.5 -3.03990923648036
-0.4 -2.61265060394064
-0.3 -2.31396247937786
-0.2 -2.09106750861755
-0.1 -1.91702726979501
0 -1.77652434594016
0.1 -1.66015373663123
0.2 -1.56179649850968
0.3 -1.47728664184011
0.4 -1.4036818236852
0.5 -1.33883997851819
0.6 -1.2811613831004
0.7 -1.22942506310671
0.8 -1.18268144076259
0.9 -1.14017979732018
1 -1.10131799747731
1.1 -1.06560685686446
1.2 -1.03264438356538
1.3 -1.00209682626611
1.4 -0.973684507888262
1.5 -0.947171083865313
1.6 -0.922355290798424
1.7 -0.899064532681888
1.8 -0.877149841174214
1.9 -0.856481875927772
2 -0.836947721062818
2.1 -0.818448297428081
2.2 -0.800896255746338
2.3 -0.784214248663571
2.4 -0.768333503839965
2.5 -0.753192638084911
2.6 -0.738736665903059
2.7 -0.724916165912033
2.8 -0.711686576283305
2.9 -0.699007596266688
3 -0.686842675434512
3.1 -0.675158575850804
3.2 -0.663924995175062
3.3 -0.653114240927836
3.4 -0.642700947901767
3.5 -0.632661832137248
3.6 -0.622975475980792
3.7 -0.61362213970048
3.8 -0.604583595859665
3.9 -0.595842983268828
4 -0.587384677836581
4.1 -0.579194178055134
4.2 -0.57125800319864
4.3 -0.56356360259831
4.4 -0.55609927459637
4.5 -0.548854093980785
4.6 -0.541817846870597
4.7 -0.534980972163507
4.8 -0.528334508777425
4.9 -0.52187004801971
5 -0.515579690504711
5.1 -0.509456007114546
5.2 -0.503492003561702
5.3 -0.497681088166739
5.4 -0.492017042511618
5.5 -0.486493994669792
5.6 -0.481106394749607
5.7 -0.475848992518116
5.8 -0.47071681689908
5.9 -0.465705157162223
6 -0.460809545641065
6.1 -0.456025741834514
6.2 -0.451349717762975
6.3 -0.446777644463518
6.4 -0.442305879520707
6.5 -0.437930955540411
6.6 -0.433649569483355
6.7 -0.429458572783548
6.8 -0.425354962184148
6.9 -0.421335871229917
7 -0.417398562361364
7.1 -0.413540419560791
7.2 -0.409758941505306
7.3 -0.406051735185886
7.4 -0.402416509955476
7.5 -0.398851071972359
7.6 -0.395353319008154
7.7 -0.3919212355924
7.8 -0.388552888468295
7.9 -0.385246422336188
8 -0.382000055863566
8.1 -0.378812077942009
8.2 -0.375680844173199
8.3 -0.372604773567623
8.4 -0.369582345440886
8.5 -0.366612096493782
8.6 -0.363692618063411
8.7 -0.360822553533604
8.8 -0.358000595893841
8.9 -0.355225485436705
9 -0.352496007584671
9.1 -0.349810990837693
9.2 -0.347169304833777
9.3 -0.344569858515202
9.4 -0.342011598393694
9.5 -0.33949350690827
9.6 -0.337014600869959
9.7 -0.334573929988032
9.8 -0.332170575472715
9.9 -0.329803648709749
};
\addplot [line width=0.2mm, black!60, dash pattern=on 2pt off 2pt]
table {%
-0.8 -12.2710280539585
-0.7 -6.76145537458028
-0.6 -5.2058003291755
-0.5 -4.34919385558753
-0.4 -3.78523641437757
-0.3 -3.3784462892193
-0.2 -3.06781809536559
-0.1 -2.82106944166339
0 -2.6192604201916
0.1 -2.45044351604206
0.2 -2.30666284994123
0.3 -2.18239231784338
0.4 -2.07366106295662
0.5 -1.97753487192723
0.6 -1.89179401388836
0.7 -1.8147252640513
0.8 -1.74498321488294
0.9 -1.68149518759895
1 -1.62339444863338
1.1 -1.56997230562832
1.2 -1.52064309874727
1.3 -1.47491818703542
1.4 -1.43238632801088
1.5 -1.39269867845933
1.6 -1.35555718682029
1.7 -1.32070550941358
1.8 -1.28792182865234
1.9 -1.25701312131318
2 -1.22781054417211
2.1 -1.20016568917253
2.2 -1.17394752147103
2.3 -1.14903985834751
2.4 -1.12533927990432
2.5 -1.10275338703553
2.6 -1.08119934063339
2.7 -1.0606026300414
2.8 -1.0408960295213
2.9 -1.02201870980899
3 -1.00391547829856
3.1 -0.986536126458639
3.2 -0.96983486708065
3.3 -0.953769847131314
3.4 -0.938302724514207
3.5 -0.923398299081391
3.6 -0.909024189878034
3.7 -0.895150551938474
3.8 -0.881749827039164
3.9 -0.868796523707274
4 -0.856267022517093
4.1 -0.844139403313837
4.2 -0.832393291509315
4.3 -0.821009721013215
4.4 -0.809971011715871
4.5 -0.799260659733475
4.6 -0.78886323887563
4.7 -0.778764312005279
4.8 -0.768950351139339
4.9 -0.759408665290683
5 -0.750127335180849
5.1 -0.741095154064466
5.2 -0.732301574001026
5.3 -0.723736656991524
5.4 -0.71539103046844
5.5 -0.707255846688295
5.6 -0.699322745628734
5.7 -0.691583821039242
5.8 -0.684031589332072
5.9 -0.676658961041671
6 -0.669459214595888
6.1 -0.6624259721942
6.2 -0.655553177585375
6.3 -0.648835075575445
6.4 -0.642266193107502
6.5 -0.635841321773609
6.6 -0.629555501631976
6.7 -0.623404006216202
6.8 -0.617382328634115
6.9 -0.611486168663421
7 -0.605711420763774
7.1 -0.600054162921623
7.2 -0.594510646270732
7.3 -0.589077285420321
7.4 -0.583750649429669
7.5 -0.578527453389154
7.6 -0.573404550551718
7.7 -0.568378924976629
7.8 -0.563447684645885
7.9 -0.558608055017309
8 -0.553857372982549
8.1 -0.549193081200172
8.2 -0.544612722776215
8.3 -0.540113936268412
8.4 -0.535694450988825
8.5 -0.531352082587586
8.6 -0.527084728894552
8.7 -0.522890366004395
8.8 -0.518767044582109
8.9 -0.514712886388459
9 -0.510726080991009
9.1 -0.506804882664818
9.2 -0.502947607462512
9.3 -0.499152630444707
9.4 -0.495418383060542
9.5 -0.491743350668638
9.6 -0.488126070189674
9.7 -0.484565127882668
9.8 -0.481059157236971
9.9 -0.477606836973149
};
\addplot [line width=0.2mm, black!60, dash pattern=on 2pt off 2pt]
table {%
-25 0.45179850964753
-24.9 0.455107835804129
-24.8 0.458473572470679
-24.7 0.461897336051681
-24.6 0.465380809370868
-24.5 0.46892574523345
-24.4 0.472533970225591
-24.3 0.476207388766518
-24.2 0.479947987436053
-24.1 0.483757839597664
-24 0.487639110342556
-23.9 0.491594061780101
-23.8 0.495625058704317
-23.7 0.499734574667369
-23.6 0.503925198495878
-23.5 0.508199641286586
-23.4 0.512560743925868
-23.3 0.517011485176545
-23.2 0.521554990385915
-23.1 0.526194540868787
-23 0.530933584030736
-22.9 0.535775744297125
-22.8 0.540724834928127
-22.7 0.545784870801851
-22.6 0.550960082262765
-22.5 0.556254930132623
-22.4 0.561674122018284
-22.3 0.56722263002524
-22.2 0.572905710039577
-22.1 0.578728922734456
-22 0.584698156487172
-21.9 0.590819652412113
-21.8 0.597100031742307
-21.7 0.603546325821053
-21.6 0.610166008999717
-21.5 0.61696703477386
-21.4 0.623957875545924
-21.3 0.631147566427814
-21.2 0.638545753595452
-21.1 0.646162747740083
-21 0.654009583262862
-20.9 0.662098083945744
-20.8 0.670440935941382
-20.7 0.679051769054994
-20.6 0.687945247437123
-20.5 0.697137170988752
-20.4 0.706644588983498
-20.3 0.716485927664693
-20.2 0.726681133864531
-20.1 0.737251837048244
-20 0.748221532603413
-19.9 0.759615789702071
-19.8 0.771462487671874
-19.7 0.783792085551365
-19.6 0.796637930403466
-19.5 0.810036611057624
-19.4 0.824028365300395
-19.3 0.838657550194713
-19.2 0.853973187270346
-19.1 0.870029596896889
-19 0.886887139368397
-18.9 0.904613084302949
-18.8 0.923282635087499
-18.7 0.942980141722178
-18.6 0.963800543857188
-18.5 0.985851096818141
-18.4 1.00925344775839
-18.3 1.03414614798988
-18.2 1.06068771268511
-18.1 1.08906037290091
-18 1.11947471066604
-17.9 1.1521754306921
-17.8 1.18744860947603
-17.7 1.22563088522899
-17.6 1.26712122705057
-17.5 1.31239617526729
-17.4 1.36202981833539
-17.3 1.4167203322796
-17.2 1.47732576738164
-17.1 1.54491311283971
-17 1.62082683416737
-16.9 1.70678665929195
-16.8 1.80503051262692
-16.7 1.91852936800717
-16.6 2.05132095579819
-16.5 2.20904863750002
-16.4 2.3998735574708
-16.3 2.63611142083118
-16.2 2.93739661400461
-16.1 3.33743456961115
-16 3.90055872320055
-15.9 4.77197661075984
-15.8 6.40245832837376
-15.7 13.4178199965611
};
\addplot [line width=0.2mm, black!60, dash pattern=on 2pt off 2pt]
table {%
-25 0.160010721357045
-24.9 0.160992312255752
-24.8 0.161986605598819
-24.7 0.162993858984882
-24.6 0.164014337202332
-24.5 0.16504831248603
-24.4 0.166096064783499
-24.3 0.167157882033849
-24.2 0.168234060458124
-24.1 0.169324904862905
-24 0.170430728956834
-23.9 0.171551855681787
-23.8 0.172688617558559
-23.7 0.173841357048445
-23.6 0.175010426930317
-23.5 0.176196190698605
-23.4 0.177399022970862
-23.3 0.178619309927264
-23.2 0.179857449759032
-23.1 0.181113853144345
-23 0.182388943743838
-22.9 0.18368315872265
-22.8 0.184996949294592
-22.7 0.186330781294543
-22.6 0.187685135778195
-22.5 0.189060509651372
-22.4 0.190457416330094
-22.3 0.191876386434866
-22.2 0.19331796851697
-22.1 0.194782729824994
-22 0.196271257107571
-21.9 0.197784157458763
-21.8 0.199322059206098
-21.7 0.200885612845416
-21.6 0.202475492023955
-21.5 0.204092394575537
-21.4 0.205737043610635
-21.3 0.207410188664148
-21.2 0.209112606905168
-21.1 0.210845104411745
-21 0.212608517514086
-20.9 0.214403714215399
-20.8 0.216231595684054
-20.7 0.218093097833592
-20.6 0.219989192989513
-20.5 0.221920891648433
-20.4 0.223889244337707
-20.3 0.225895343579256
-20.2 0.227940325966321
-20.1 0.230025374358554
-20 0.232151720203859
-19.9 0.234320645994431
-19.8 0.236533487865765
-19.7 0.238791638347576
-19.6 0.241096549275909
-19.5 0.243449734877246
-19.4 0.24585277503456
-19.3 0.248307318747661
-19.2 0.250815087799437
-19.1 0.253377880641247
-19 0.255997576510606
-18.9 0.258676139796753
-18.8 0.261415624667714
-18.7 0.264218179976813
-18.6 0.267086054464556
-18.5 0.270021602273741
-18.4 0.27302728879767
-18.3 0.276105696880078
-18.2 0.279259533388397
-18.1 0.282491636181344
-18 0.285804981494342
-17.9 0.289202691765001
-17.8 0.292688043924724
-17.7 0.296264478180421
-17.6 0.299935607313099
-17.5 0.303705226520497
-17.4 0.30757732383105
-17.3 0.311556091117962
-17.2 0.315645935742761
-17.1 0.319851492857594
-17 0.324177638397712
-16.9 0.328629502794762
-16.8 0.333212485444179
-16.7 0.337932269959002
-16.6 0.342794840246663
-16.5 0.347806497445074
-16.4 0.352973877757759
-16.3 0.358303971236113
-16.2 0.363804141552679
-16.1 0.369482146829064
-16 0.375346161585099
-15.9 0.381404799889889
-15.8 0.387667139818759
-15.7 0.394142749337984
-15.6 0.400841713775782
-15.5 0.407774665072454
-15.4 0.414952813055769
-15 0.446363756329006
-14.9 0.454958779860659
-14.8 0.463880496521499
-14.7 0.473145294464958
-14.6 0.482770520318202
-14.5 0.492774555760659
-14.4 0.503176905171395
-14.3 0.513998296486033
-14.2 0.525260797770298
-14.1 0.536987952434713
-14 0.549204936489169
-13.9 0.561938741767839
-13.8 0.575218389689398
-13.7 0.589075180818056
-13.6 0.603542986353686
-13.5 0.618658588686662
-13.4 0.63446207938375
-13.3 0.650997324486015
-13.2 0.668312508885712
-13.1 0.686460773915158
-13 0.705500965306646
-12.9 0.725498512492976
-12.8 0.74652646518361
-12.7 0.768666719485646
-12.6 0.79201147409875
-12.5 0.816664967856888
-12.4 0.84274556398497
-12.3 0.870388265046676
-12.2 0.899747767321775
-12.1 0.931002196593633
-12 0.964357712395262
-11.9 1.00005422952872
-11.8 1.03837259133519
-11.7 1.07964364954637
-11.6 1.12425987704104
-11.5 1.17269038799931
-11.4 1.2255006051164
-11.3 1.28337836101478
-11.2 1.34716905855648
-11.1 1.41792382568999
-11 1.49696670490242
-10.9 1.5859903936461
-10.8 1.68719598426328
-10.7 1.80350265975419
-10.6 1.9388727388598
-10.5 2.09883530015481
-10.4 2.29136990475241
-10.3 2.5284864833547
-10.2 2.82926478911821
-10.1 3.22629838848927
-10 3.78134609543461
-9.9 4.63210516167276
-9.8 6.19491550748911
-9.7 11.9733854526646
};
\addplot [line width=0.2mm, black!60, dash pattern=on 2pt off 2pt]
table {%
-15 -7.78270231539922
-14.9 -4.76515898309123
-14.8 -3.5379114236493
-14.7 -2.81637277676759
-14.6 -2.32908913244441
-14.5 -1.97373806101016
-14.4 -1.7013395405207
-14.3 -1.48498814081643
-14.2 -1.3084950980564
-14.1 -1.16146212356292
-14 -1.03687157964212
-13.9 -0.929799389454067
-13.8 -0.836679731885418
-13.7 -0.75486191233669
-13.6 -0.682331424123338
-13.5 -0.61752788053228
-13.4 -0.55922247107518
-13.3 -0.506433278464611
-13.2 -0.458365403589442
-13.1 -0.414367771937868
-13 -0.373901415130769
-12.9 -0.336515806941536
-12.8 -0.301830955484658
-12.7 -0.269523676022756
-12.6 -0.239316944590113
-12.5 -0.210971552014498
-12.4 -0.184279496206356
-12.3 -0.15905870223171
-12.2 -0.135148766629442
-12.1 -0.112407498901627
-12 -0.0907080884845125
-11.9 -0.0699367660908347
-11.8 -0.0499908583812818
-11.7 -0.0307771574241631
-11.6 -0.0122105434006645
-11.5 0.00578718803893426
-11.4 0.0232883324213331
-11.3 0.0403601480676114
-11.2 0.0570656176735281
-11.1 0.07346412017129
-11 0.0896120295204271
-10.9 0.105563254181177
-10.8 0.121369728738121
-10.7 0.137081867330931
-10.6 0.152748987118873
-10.5 0.168419708877289
-10.4 0.184142340940662
-10.3 0.199965252030813
-10.2 0.215937237999211
-10.1 0.232107887158462
-10 0.24852794868498
-9.9 0.26524970848805
-9.8 0.282327377072035
-9.7 0.299817494185362
-9.6 0.317779355585635
-9.5 0.336275468063876
-9.4 0.355372040077539
-9 0.439248318397618
-8.9 0.462510927776328
-8.8 0.48688383126518
-8.7 0.512478577411451
-8.6 0.539417469228547
-8.5 0.567835310273567
-8.4 0.597881546587331
-8.3 0.629722916908117
-8.2 0.663546755481437
-8.1 0.699565133628835
-8 0.738020082122043
-7.9 0.779190212651922
-7.8 0.823399162925087
-7.7 0.87102644095138
-7.6 0.922521462803584
-7.5 0.978421900412144
-7.4 1.03937793921232
-7.3 1.10618478353913
-7.2 1.17982689829435
-7.1 1.26153931149942
-7 1.35289431180338
-6.9 1.45592696230277
-6.8 1.57332176391825
-6.7 1.70869907406516
-6.6 1.86707108044812
-6.5 2.05560049771727
-6.4 2.28493326725594
-6.3 2.57170510886673
-6.2 2.9436965325702
-6.1 3.4518202336475
-6 4.20356016976705
-5.9 5.48983914748796
-5.8 8.78230408800556
};
\addplot [line width=0.2mm, black!60, dash pattern=on 2pt off 2pt]
table {%
-9 -8.17837874316797
-8.9 -4.90032974744227
-8.8 -3.62223839415906
-8.7 -2.8778035740928
-8.6 -2.37661616656173
-8.5 -2.01131542953125
-8.4 -1.73102432024305
-8.3 -1.50796840919315
-8.2 -1.32550479336322
-8.1 -1.17297270572345
-8 -1.0431891408798
-7.9 -0.931117067612954
-7.8 -0.833107156741678
-7.7 -0.746442054799154
-7.6 -0.669050166605109
-7.5 -0.599319198337208
-7.4 -0.535970867077565
-7.3 -0.477974433177719
-7.2 -0.424485610687253
-7.1 -0.37480249101168
-7 -0.32833312109623
-6.9 -0.284571213316523
-6.8 -0.243077617037686
-6.7 -0.203465923605906
-6.6 -0.165391064245706
-6.5 -0.128540086896728
-6.4 -0.0926245201035988
-6.3 -0.0573738847471289
-6.2 -0.0225300201671065
-6.1 0.0121580355434064
-6 0.0469388210551188
-5.9 0.0820630597833645
-5.8 0.117788256560683
-5.7 0.154383429637762
-5.6 0.192134125226515
-5.5 0.231347873618976
-5.1 0.410120705181088
-5 0.462515659903529
-4.9 0.51910251699393
-4.8 0.58059175654399
-4.7 0.647819383623781
-4.6 0.721782730210276
-4.5 0.803690598569816
-4.4 0.895034917443084
-4.3 0.99769513453944
-4.2 1.1140934640847
-4.1 1.24743142507106
-4 1.40206125084086
-3.9 1.58409169284049
-3.8 1.8024249170056
-3.7 2.07064341117277
-3.6 2.41072718300476
-3.5 2.86119926788457
-3.4 3.49790015658613
-3.3 4.50129177419861
-3.2 6.50315122208509
};
\addplot [line width=0.2mm, black!60, dash pattern=on 2pt off 2pt]
table {%
-5.1 -11.3004377954798
-5 -5.54070049205998
-4.9 -3.97282531300889
-4.8 -3.11592880083239
-4.7 -2.55454236685049
-4.6 -2.15115392329575
-4.5 -1.84404655913909
-4.4 -1.60060964935675
-4.3 -1.40170698692708
-4.2 -1.23523763327814
-4.1 -1.09309724159774
-4 -0.969597763140107
-3.9 -0.8605828302749
-3.8 -0.762902624991376
-3.7 -0.674086210902069
-3.6 -0.59212740525252
-3.5 -0.515337988695226
-3.4 -0.442241373772354
-3.3 -0.371490181724252
-3.2 -0.301796803042351
-3.1 -0.23186899637853
-3 -0.160343919981569
-2.9 -0.0857141557361271
-2.8 -0.00623841142696538
-2.4 0.410760347675084
-2.3 0.559573619850788
-2.2 0.73929898783585
-2.1 0.961426320734427
-2 1.24338987474296
-1.9 1.61369432330768
-1.8 2.12401401272894
-1.7 2.88324871502327
-1.6 4.18772666180333
-1.5 7.56765736134474
};
\addplot [line width=0.2mm, black!60, dash pattern=on 2pt off 2pt]
table {%
-2.4 -7.0416355139514
-2.3 -4.7149805350599
-2.2 -3.62583244599472
-2.1 -2.9544209150678
-2 -2.4897784855038
-1.9 -2.14690004905436
-1.8 -1.88349549319763
-1.7 -1.67555628079673
-1.6 -1.5081042727586
-1.5 -1.37111868607633
-1.4 -1.25754929094007
-1.3 -1.16226138085458
-1.2 -1.08142766211419
-0.8 -0.853810407270782
-0.7 -0.812956970290579
-0.6 -0.776470081322293
-0.5 -0.743673890013228
-0.4 -0.714021773667113
-0.3 -0.687067998239866
-0.2 -0.662446211375237
-0.1 -0.639853006568254
0 -0.619035273025133
0.1 -0.599780388831966
0.2 -0.581908563258257
0.3 -0.565266814242248
0.4 -0.54972419841377
0.5 -0.535168007057651
0.6 -0.521500712059109
0.7 -0.508637498039264
0.8 -0.496504255675104
0.9 -0.48503594018044
1 -0.474175220720616
1.1 -0.46387136302447
1.2 -0.454079300001785
1.3 -0.444758854796012
1.4 -0.43587408808987
1.5 -0.427392747240548
1.6 -0.419285799269116
1.7 -0.411527033242721
1.8 -0.404092720360372
1.9 -0.396961322215655
2 -0.390113239473031
2.1 -0.38353059457663
2.2 -0.377197043234218
2.3 -0.371097610323547
2.4 -0.365218546602994
2.5 -0.359547203206818
2.6 -0.354071921395299
2.7 -0.34878193543272
2.8 -0.343667286798238
2.9 -0.338718748209806
3 -0.333927756170017
3.1 -0.329286350933255
3.2 -0.32478712295364
3.3 -0.320423165005858
3.4 -0.316188029291342
3.5 -0.312075688916127
3.6 -0.308080503245435
3.7 -0.304197186666987
3.8 -0.300420780384034
3.9 -0.296746626894983
4 -0.293170346863758
4.1 -0.289687818121655
4.2 -0.286295156572435
4.3 -0.282988698800795
4.4 -0.279764986207962
4.5 -0.276620750518594
4.6 -0.273552900521245
4.7 -0.270558509920478
4.8 -0.267634806192346
4.9 -0.264779160346588
5 -0.26198907751005
5.1 -0.259262188254471
5.2 -0.256596240600064
5.3 -0.253989092633784
5.4 -0.251438705687041
5.5 -0.248943138023461
5.6 -0.246500538992535
5.7 -0.244109143608466
5.8 -0.241767267518801
5.9 -0.2394733023297
6 -0.237225711258631
6.1 -0.235023025087363
6.2 -0.232863838391585
6.3 -0.230746806023409
6.4 -0.228670639833427
6.5 -0.226634105598324
6.6 -0.224636020158669
6.7 -0.222675248735312
6.8 -0.220750702417159
6.9 -0.218861335806973
7 -0.217006144813176
7.1 -0.215184164577282
7.2 -0.213394467527092
7.3 -0.211636161546827
7.4 -0.209908388255847
7.5 -0.2082103213886
7.6 -0.206541165269482
7.7 -0.204900153370897
7.8 -0.20328654696296
7.9 -0.201699633829166
8 -0.200138727058961
8.1 -0.198603163905973
8.2 -0.197092304708479
8.3 -0.195605531868248
8.4 -0.194142248884078
8.5 -0.192701879436807
8.6 -0.191283866522077
8.7 -0.189887671628441
8.8 -0.188512773957753
8.9 -0.18715866968544
9 -0.185824871258007
9.1 -0.184510906726025
9.2 -0.183216319110069
9.3 -0.18194066579783
9.4 -0.180683517970807
9.5 -0.179444460058522
9.6 -0.178223089219126
9.7 -0.177019014844426
9.8 -0.175831858088376
9.9 -0.17466125141758
};
\end{axis}

\end{tikzpicture}